\documentclass[oneside,english]{amsart}
\usepackage[T1]{fontenc}
\usepackage[latin9]{inputenc}
\usepackage{amsthm}
\usepackage{amstext}
\usepackage{amssymb}
\usepackage{esint}
\usepackage{graphicx}
\usepackage{enumerate}
\usepackage{amscd}

\makeatletter
\newtheorem{theorem}{Theorem}[section]
\newtheorem{lemma}[theorem]{Lemma}

\newtheorem{corollary}[theorem]{Corollary}
\numberwithin{figure}{section}
\theoremstyle{definition}

\theoremstyle{remark}
\newtheorem{remark}[theorem]{Remark}

\numberwithin{equation}{section}
\makeatother

\usepackage{color}

\newcommand{\Om} {\Omega}

\usepackage{babel}

	\DeclareMathOperator{\loc}{loc}

\newcommand\R{\mathbb R}

\begin{document}

\title[On the Neumann $(p,q)$-eigenvalue problem in H\"older singular domains]
{On the Neumann $(p,q)$-eigenvalue problem in H\"older singular domains}

\author{Prashanta Garain, Valerii Pchelintsev, Alexander Ukhlov}

\begin{abstract}
In the article we study the Neumann $(p,q)$-eigenvalue problems
in bounded H\"older $\gamma$-singular domains $\Omega_{\gamma}\subset \mathbb{R}^n$. In the case $1<p<\infty$ and $1<q<p^{*}_{\gamma}$ we prove
solvability of this eigenvalue problem and existence of the minimizer of the
associated variational problem. In addition, we establish some regularity results of the eigenfunctions and some estimates of $(p,q)$-eigenvalues.
\end{abstract}
\maketitle
\footnotetext{\textbf{Key words and phrases:} $p$-Laplacian, Neumann eigenvalue problem, existence, regularity, quasiconformal mappings.}
\footnotetext{\textbf{2020
Mathematics Subject Classification:} 35P15, 35P30, 35A01, 35J92, 46E35, 30C65.}

\section{Introduction}

In this article, we investigate the following Neumann $(p,q)$-eigenvalue problem:
\begin{equation}\label{maineqn}
-\Delta_p u:=-\text{div}(|\nabla u|^{p-2}\nabla u)=\lambda\|u\|_{L^q(\Omega_{\gamma})}^{p-q}|u|^{q-2}u\text{ in }\Omega_{\gamma},\quad\frac{\partial u}{\partial\nu}=0\text{ on }\partial \Omega_{\gamma},
\end{equation}
in bounded domains with anisotropic H\"older $\gamma$-singularities
$$
\Omega_{\gamma}=\{ x=(x_1,x_2,\ldots,x_n)\in\mathbb R^n : 0<x_n<1, 0<x_i<g_i(x_n),\,i=1,2,\dots,n-1\},
$$
where $g_i(t)=t^{\gamma_i}$, $\gamma_i\geq 1$, $0< t< 1$, are H\"older functions and we denote by $\gamma={\log (g_1(t)\cdot ... \cdot g_{n-1}(t))}{(\log t)}^{-1}+1$.
It is evident that $\gamma\geq n$. In the case $g_1=g_2=\dots=g_{n-1}$ we will say that domain $\Omega_{\gamma}$ is a domain with $\sigma$-H\"older singularity, $\sigma=(\gamma-1)/(n-1)$.
Note that the main results of the article include the case of Lipschitz domains $\Omega\subset\mathbb R^n$, in this case we can put $\gamma=n$.
We assume that $1<p<\gamma$ and $1<q<p^*_{\gamma}$, where $p^*_{\gamma}={\gamma p}/{(\gamma -p)}$.

The associated Dirichlet eigenvalue problem of \eqref{maineqn}, i.e.,
\begin{equation}\label{Deqn}
-\Delta_p u=\lambda\|u\|_{L^q(\Om)}^{p-q}|u|^{q-2}u\text{ in }\Om,\quad u=0\text{ on }\partial\Om,
\end{equation}
is widely studied in the literature concerning the existence, regularity among other qualitative properties of eigenfunctions for $p=q$ as well as for $p\neq q$. When $p=q$ in \eqref{Deqn}, we refer to Lindqvist \cite{Lind92}, Garc\'{\i}a Azorero-Peral Alonso \cite{GPe}, Anane-Tousli \cite{AnT, An}, Dr\'abek-Kufner-Nicolosi \cite{Drabek} and the references therein. When $p\neq q$ in \eqref{Deqn}, we refer to Franzina-Lamberti \cite{Franzina}, Ercole \cite{Ercole}, Garain-Ukhlov \cite{GUsub} and the references therein.

In contrast to the Dirichlet problem \eqref{Deqn}, the Neumann problem \eqref{maineqn} is less understood in the literature. In this concern, for $p=q$ in \eqref{maineqn}, i.e., for the equation
\begin{equation}\label{Npeqn}
-\Delta_p u=\lambda|u|^{p-2}u\text{ in }\Om,\quad\frac{\partial u}{\partial\nu}=0\text{ on }\partial\Om,
\end{equation}
in Lipschitz domains $\Omega\subset\mathbb R^n$ we refer to Greco-Lucia \cite{Greco} for the linear case $p=q=2$. In the nonlinear setting, L\^{e} \cite{Anle} proved existence, regularity among other qualitative properties of the eigenfunctions of \eqref{Npeqn}. For $p\neq q$ in \eqref{maineqn} under the hypothesis that $\Omega=B_1\cup B_2$ for two disjoint balls $B_1$ and $B_2$, Croce-Henrot-Pisante \cite{Croce} proved existence of a bounded eigenfunction of \eqref{maineqn} in the space $W_{\text{per}}^{1,p}(\Omega)$, which stands for the Sobolev space of functions in $W^{1,p}(\Omega)$ taking constant boundary values, see also Nazarov \cite{Nazarov} for related results.

Estimates of Neumann eigenvalues of the $p$-Laplace operator in  non-convex domains is a long-standing complicated problem \cite{ENT,PW,PS51}. This problem was partially solved on the base of composition operators on Sobolev spaces, see for example, \cite{GPU18,GU16,GU17}. In this article we prove that the first non-trivial Neumann $(p,q)$-eigenvalues can be characterized by the Min-Max Principle
$$
\lambda_{p,q}(\Omega_{\gamma})\\=\inf \left\{\frac{\|\nabla u\|_{L^p(\Omega_{\gamma})}^p}{\|u\|_{L^q(\Omega_{\gamma})}^p} :
u \in W^{1,p}(\Omega_{\gamma}) \setminus \{0\}, \int_{\Omega_{\gamma}} |u|^{q-2}u~dx=0 \right\}.
$$

By using this characterization we give lower estimates of Neumann $(p,q)$-eigen\-va\-lu\-es in bounded H\"older singular domains $\Omega_{\gamma}\subset\mathbb R^n$. In particular, we give the following lower estimate of the first nontrivial Neumann $(p,q)$-eigenvalue
\[
{\lambda_{3,2}(\Omega_{\gamma})} \geq \left(12\pi\sqrt{(\gamma_1-1)^2+(\gamma_2-1)^2+3}\right)^{-3},
\]
where $\Omega_{\gamma}\subset\mathbb R^3$ is a domain with anisotropic H\"older $\gamma$-singularities, $\gamma=\gamma_1+\gamma_2+1$, $3<\gamma<5$.

So, this example gives, in particular, in the frameworks the conjecture by V.~Maz'\-ya \cite{M}, that exact Poincar\'e--Sobolev constants in anisotropic H\"older singular domains depends not on $\gamma=\gamma_1+...+\gamma_{n-1}+1$ only.

In addition, in any bounded H\"older singular domain $\Omega_{\gamma}$, for $q=2$, we prove existence of eigenfunction of \eqref{maineqn} in $W^{1,p}(\Omega_{\gamma})$ (see Theorems \ref{newthm}-\ref{subopthm1}) with zero mean, but not necessarily takes the constant boundary values as in \cite{Croce}. Further, we study the associated minimizing problem of \eqref{maineqn} (see Theorem \ref{subopthm1}) and prove the boundedness of the eigenfunctions among other qualitative properties (see Theorem \ref{regthm1}). To this end, we follow the approach from Ercole \cite{Ercole}.

This article is organized as follows: In Section $2$, we prove Min-Max Principle for the first non-trivial Neumann $(p,q)$-eigenvalue.
In Section $3$, we give estimates of Neumann $(p,q)$-eigenvalues in H\"older singular domains.
Finally, in Section $4$, we mention the functional setting related to the problem \eqref{maineqn} and further state and prove our regularity results.

\section{Neumann eigenvalue problem}

\subsection{Sobolev spaces}
Let $\Omega\subset\mathbb R^n$ be an open set. Then the Sobolev space $W^{1,p}(\Omega)$, $1\leq p\leq\infty$,
is defined as a Banach space of locally integrable weakly differentiable functions
$u:\Omega\to\mathbb{R}$ equipped with the following norm:
\[
\\\|u\|_{W^{1,p}(\Omega)}:=\|u\|_{L^p(\Omega)}+\|\nabla u\|_{L^p(\Omega)},
\]
where $L^p(\Omega)$ is the Lebesgue space with the standard norm.
In accordance with the non-linear potential theory \cite{HKM,MH72} we consider elements of Sobolev spaces $W^{1,p}(\Omega)$ as equivalence classes up to a set of $p$-capacity zero  \cite{M}.

The Sobolev space $W^{1,p}_{\loc}(\Omega)$ is defined as follows: $u\in W^{1,p}_{\loc}(\Omega)$
if and only if $u\in W^{1,p}(U)$ for every open and bounded set $U\subset  \Omega$ such that
$\overline{U}  \subset \Omega$, where $\overline{U} $ is the closure of the set $U$.

We consider domains with anisotropic H\"older $\gamma$-sin\-gu\-larities
$$
\Omega_{\gamma}=\{ x=(x_1,x_2,\cdots,x_n)\in\mathbb R^n : 0<x_n<1, 0<x_i<g_i(x_n),\,i=1,2,\dots,n-1\},
$$
where $g_i(t)=t^{\gamma_i}$, $\gamma_i\geq 1$, $0< t< 1$ are H\"older functions and we denote by $\gamma={\log (g_1(t)\cdot ... \cdot g_{n-1}(t))}{(\log t)}^{-1}+1$.

Let us recall the Sobolev embedding theorem in H\"older singular domains \cite{GG94,GU09}.

\begin{theorem}\label{thmemb}
Let $\Omega_{\gamma}\subset\mathbb R^n$ be a domain with anisotropic H\"older $\gamma$-sin\-gu\-larities. Suppose $1<p<\gamma$. Then the embedding operator
$$
W^{1,p}(\Omega_{\gamma})\hookrightarrow L^{r}(\Omega_{\gamma})
$$
is compact for any $1<r<p_{\gamma}^*$, where $p_{\gamma}^*={\gamma p}/{(\gamma -p)}$.
\end{theorem}

Throughout Sections 2-3, we assume that $1<p<\gamma$ and $1<q<p_{\gamma}^*$ unless otherwise mentioned.
Now for $\lambda\in \R$, we consider the following Neumann $(p,q)$-eigenvalue problem
\begin{equation}\label{meqn}
-\Delta_p u:=-\text{div}(|\nabla u|^{p-2}\nabla u)=\lambda\|u\|_{L^{q}(\Omega_{\gamma})}^{p-q}|u|^{q-2}u\text{ in }\Omega_{\gamma},\quad\frac{\partial u}{\partial\nu}=0\text{ on }\partial \Omega_{\gamma},
\end{equation}
where $\Omega_{\gamma}\subset\mathbb{R}^n$ and $\nu$ is the outward unit normal to $\partial \Omega_{\gamma}$ in the weak formulation:

\vskip 0.2cm

\noindent
{\it The pair $(\lambda,u)\in \R\times W^{1,p}(\Omega_{\gamma})\setminus\{0\}$ is an eigenpair of \eqref{maineqn} if for every $\phi\in W^{1,p}(\Omega_{\gamma})$, we have
\begin{equation}\label{mwksol}
\begin{split}
\int_{\Omega_{\gamma}}|\nabla u|^{p-2}\nabla u\nabla\phi\,dx=\lambda\|u\|_{L^{q}(\Omega_{\gamma})}^{p-q}\int_{\Omega_{\gamma}}|u|^{q-2}u\phi\,dx.
\end{split}
\end{equation}
We refer to $\lambda$ as an eigenvalue and $u$ as the eigenfunction corresponding to $\lambda$.}

\vskip 0.2cm

Now we prove the Min-Max Principle for the first non-trivial Neumann $(p,q)$-eigenvalue.
We establish Theorem \ref{minthm} following the proof of \cite[Lemma 2]{Croce}. To this end, first we obtain the following auxiliary result.
\begin{lemma}\label{lem1}
Let $v\in W^{1,p}(\Omega_{\gamma})\setminus\{0\}$ be such that $\int_{\Omega_{\gamma}}|v|^{q-2}v\,dx=0$. Then there exists a constant $C=C(\Omega_{\gamma})>0$ such that
$$
\|v\|_{L^p(\Omega_{\gamma})}\leq C\|\nabla v\|_{L^p(\Omega_{\gamma})}.
$$
\end{lemma}

\begin{proof}
By contradiction, suppose for every $n\in\mathbb{N}$, there exists $v_n\in W^{1,p}(\Omega_{\gamma})\setminus\{0\}$ such that $\int_{\Omega_{\gamma}}|v_n|^{q-2}v_n\,dx=0$ and
\begin{equation}\label{cn}
\|v_n\|_{L^p(\Omega_{\gamma})}>n\|\nabla v_n\|_{L^p(\Omega_{\gamma})}.
\end{equation}
Without loss of generality, let us assume that $\|v_n\|_{L^p(\Omega_{\gamma})}=1$. If not, we define
$$
u_n=\frac{v_n}{\|v_n\|_{L^p(\Omega_{\gamma})}},
$$
then $\|u_n\|_{L^p(\Omega_{\gamma})}=1$ and \eqref{cn} holds for $u_n$ and also $\int_{\Omega_{\gamma}}|u_n|^{q-2}u_n\,dx=0$. By \eqref{cn}, since $\|\nabla v_n\|_{L^p(\Omega_{\gamma})}\to 0$ as $n\to\infty$, we have that the sequence of functions $\{v_n\}, n\in\mathbb{N}$, is uniformly bounded in $W^{1,p}(\Omega_{\gamma})$. By Theorem \ref{thmemb}, since the embedding operator
$$
W^{1,p}(\Omega_{\gamma})\hookrightarrow L^q(\Omega_{\gamma}),\quad 1<q<p_{\gamma}^*,
$$
is compact, there exists $v\in W^{1,p}(\Omega_{\gamma})$ such that
$$
v_n{\rightharpoonup} v\text{ weakly\,in }W^{1,p}(\Omega_{\gamma}),\quad v_n\to v\text{ strongly\,in }L^q(\Omega_{\gamma}),\quad\forall \,1<q<p_{\gamma}^*,
$$
and $g\in L^q(\Omega_{\gamma})$ such that
$$
|v_n|\leq g\text{ a.e.\, in }\Omega_{\gamma},\quad \nabla v_n{\rightharpoonup} \nabla v\text{ weakly\,in } L^p(\Omega_{\gamma}).
$$
Since $\|\nabla v_n\|_{L^p(\Omega_{\gamma})}\to 0$ as $n\to\infty$, we have $\nabla v_n{\rightharpoonup} 0$ weakly in $L^p(\Omega_{\gamma})$, hence $\nabla v=0$ a.e. in $\Omega_{\gamma}$, which gives that $v=constant$ a.e. in $\Omega_{\gamma}$. This combined with the fact that
$$
0=\lim_{n\to\infty}\int_{\Omega_{\gamma}}|v_n|^{q-2}v_n\,dx=\int_{\Omega_{\gamma}}|v|^{q-2}v\,dx
$$
gives that $v=0$ a.e. in $\Omega_{\gamma}$. This contradicts the hypothesis that $\|v_n\|_{L^p(\Omega_{\gamma})}=1$. This completes the proof.
\end{proof}

\begin{theorem}\label{minthm}
There exists $u\in W^{1,p}((\Omega_{\gamma})\setminus\{0\}$ such that $\int_{\Omega_{\gamma}}|u|^{q-2}u\,dx=0$. Moreover,
\begin{multline*}
\lambda_{p,q}(\Omega_{\gamma})=\inf \left\{\frac{\|\nabla u\|_{L^p(\Omega_{\gamma})}^p}{\|u\|_{L^q(\Omega_{\gamma})}^p} : u \in W^{1,p}(\Omega_{\gamma}) \setminus \{0\},
\int_{\Omega_{\gamma}} |u|^{q-2}u~dx=0 \right\}\\
=\frac{\|\nabla u\|_{L^p((\Omega_{\gamma})}^p}{\|u\|_{L^q((\Omega_{\gamma})}^p}.
\end{multline*}
Further, $\lambda_{p,q}(\Omega_{\gamma})>0.$
\end{theorem}

\begin{proof}
Let $n\in\mathbb{N}$ and define the functionals $G: W^{1,p}(\Omega_{\gamma})\to\mathbb{R}$ by
$$
G(v)=\int_{\Omega_{\gamma}}|v|^{q-2}v\,dx,
$$
and $H_\frac{1}{n}: W^{1,p}(\Omega_{\gamma})\to\mathbb{R}$ by
$$
H_\frac{1}{n}(v)=\|\nabla v\|_{L^p(\Omega_{\gamma})}^p-(\lambda_{p,q}^p+\frac{1}{n})\|v\|_{L^q(\Omega_{\gamma})}^p,
$$
where we denoted $\lambda_{p,q}(\Omega_{\gamma})$ by $\lambda_{p,q}$.
By the definition of infimum, for every $n\in\mathbb{N}$, there exists $u_n\in W^{1,p}(\Omega_{\gamma})\setminus\{0\}$ such that $\int_{\Omega_{\gamma}}|u_n|^{q-2}u_n\,dx=0$ and $H_\frac{1}{n}(u_n)<0$. Without loss of generality, let us assume that $\|\nabla u_n\|_{L^p(\Omega_{\gamma})}=1$. By Lemma \ref{lem1}, the sequence $\{u_n\}$ is uniformly bounded in $W^{1,p}(\Omega_{\gamma})$. 
Hence, by Theorem \ref{thmemb}, because the embedding operator
$$
W^{1,p}(\Omega_{\gamma})\hookrightarrow L^{q}(\Omega_{\gamma}),\quad 1<q<p_{\gamma}^*,
$$
is compact, there exists $u\in W^{1,p}(\Omega_{\gamma})$ such that $u_n{\rightharpoonup} u$ weakly in $W^{1,p}(\Omega_{\gamma})$ and $u_n\to u$ strongly in $L^q(\Omega_{\gamma})$ and there exists $g\in L^q(\Omega_{\gamma})$ such that $|u_n|\leq g$ a.e. in $\Omega_{\gamma}$ and $\nabla u_n{\rightharpoonup}\nabla u$ weakly in $L^q(\Omega_{\gamma})$.

Since $|u_n|\leq g$ a.e. in $\Omega_{\gamma}$ and $u_n\to u$ a.e. in $\Omega_{\gamma}$. Then,
$$
||u_n|^{q-2}u_n|\leq |u_n|^{q-1}\leq |g|^{q-1}\in L^{q'}(\Omega_{\gamma}).
$$
So, by the Lebesgue Dominated Convergence Theorem (see, for example, \cite{Fe69}), it follows that
$$
0=\lim_{n\to\infty}\int_{\Omega_{\gamma}}|u_n|^{q-2}u_n\,dx=\int_{\Omega_{\gamma}}|u|^{q-2}u\,dx.
$$
So, $\int_{\Omega_{\gamma}}|u|^{q-2}u\,dx=0$. Due to $H_\frac{1}{n}(u_n)<0$, we obtain
\begin{equation}\label{1}
\|\nabla u_n\|_{L^p(\Omega_{\gamma})}^p-(\lambda_{p,q}^p+\frac{1}{n})\|u_n\|_{L^q(\Omega_{\gamma})}^p<0.
\end{equation}
Since $\nabla u_n{\rightharpoonup} \nabla u$ weakly in $L^p(\Omega_{\gamma})$, by the weak lower semicontinuity of norm, we have
$$
\|\nabla u\|_{L^p(\Omega_{\gamma})}\leq \lim\inf_{n\to\infty}\|\nabla u_n\|_{L^p(\Omega_{\gamma})}=1.
$$
So, by passing to the limit in \eqref{1}, we get
$$
\lambda_{p,q}\geq \frac{\|\nabla u\|_{L^p(\Omega_{\gamma})}^p}{\|u\|_{L^q(\Omega_{\gamma})}^p}.
$$
Therefore, by the definition of $\lambda_{p,q}$, we obtain
$$
\lambda_{p,q} = \frac{\|\nabla u\|_{L^p(\Omega_{\gamma})}^p}{\|u\|_{L^q(\Omega_{\gamma})}^p}.
$$
Now, since $H_{\frac{1}{n}}(u_n)<0$ and $\|\nabla u_n\|_{L^p(\Omega_{\gamma})}=1$, we have
$$
1-(\lambda_{p,q}^p+\frac{1}{n})\|u_n\|_{L^q(\Omega_{\gamma})}^p<0.
$$
Letting $n\to\infty$, we get
$$
\|u\|_{L^q(\Omega_{\gamma})}\lambda_{p,q}\geq 1,
$$
which gives $\lambda_{p,q}>0$ and $u\neq 0$ a.e. in $\Omega_{\gamma}$.
\end{proof}

\section{Estimates of Neumann $(p,q)$-eigenvalues}
In this section we give estimates of Neumann $(p,q)$-eigenvalues in H\"older singular domains. The suggested method is based on Theorem~\ref{minthm} and on the composition operators theory on Sobolev spaces \cite{U93,VU02,VU04,VU05}.

\subsection{Composition operators on Sobolev spaces} The seminormed Sobolev space $L^{1,p}(\Omega)$ in a domain $\Omega\subset\mathbb R^n$
is the space of all locally integrable weakly differentiable functions with the following seminorm:
\[
\|f\|_{L^{1,p}(\Omega)}=\biggr(\int_{\Omega}|\nabla f(x)|^{p}\, dx\biggr)^{\frac{1}{p}}.
\]

Let $\Omega$ and $\widetilde{\Omega}$ be domains in the Euclidean space $\subset\mathbb R^n$. Then a homeomorphism
$\varphi:\Omega\to \widetilde{\Omega}$ belongs to the Sobolev class $W^{1,p}_{\loc}(\Omega)$,
$1\leq p\leq\infty$, if its coordinate functions $\varphi_j$ belong to $W^{1,p}_{\loc}(\Omega)$, $j=1,\dots,n$.
In this case the formal Jacobi matrix
$D\varphi(x)=\left(\frac{\partial \varphi_i}{\partial x_j}(x)\right)$, $i,j=1,\dots,n$,
and its determinant (Jacobian) $J(x,\varphi)=\det D\varphi(x)$ are well defined at
almost all points $x\in \Omega$. The norm $|D\varphi(x)|$ of the matrix
$D\varphi(x)$ is the norm of the corresponding linear operator $D\varphi (x):\mathbb R^n \rightarrow \mathbb R^n$ defined by the matrix $D\varphi(x)$.

Recall that a Sobolev homeomorphism $\varphi: \Omega \to \widetilde{\Omega}$ of the class $W^{1,1}_{\loc}(\Omega)$ has finite distortion \cite{VGR} if
\[
D\varphi(x)=0\,\,\, \text{a.e. on the set}\,\,\, Z=\{x \in \Omega:|J(x,\varphi)|=0\}.
\]

The result of \cite{U93} gives the analytic description of composition operators on seminormed Sobolev spaces (see, also
\cite{VU02}) and asserts that
\begin{theorem}
\label{CompTh} \cite{U93} Let $\varphi:\Omega\to\widetilde{\Omega}$ be a homeomorphism
between two domains $\Omega$ and $\widetilde{\Omega}$ $\subset\mathbb R^n$. Then $\varphi$ induces by the composition rule $\varphi^{\ast}(f)=f\circ\varphi$ a bounded composition operator
\[
\varphi^{\ast}:L^{1,p}(\widetilde{\Omega})\to L^{1,s}(\Omega),\,\,\,1\leq s< p<\infty,
\]
 if and only if $\varphi\in W_{\loc}^{1,1}(\Omega)$, has finite distortion,
and
\begin{equation}\label{kps}
K_{p,s}(\varphi;\Omega)=\left(\int_\Omega \left(\frac{|D\varphi(x)|^p}{|J(x,\varphi)|}\right)^\frac{s}{p-s}~dx\right)^\frac{p-s}{ps}<\infty.
\end{equation}
\end{theorem}

The homeomorphisms which satisfy conditions of Theorem~\ref{CompTh} are called weak $(p,s)$-quasiconformal
mappings \cite{VU98}.

\subsection{Spectral estimates}

Let $\Omega \subset \mathbb R^n$ be a bounded domain. Then $\Omega$ is called an $(r,s)$-Sobolev-Poincar\'e domain, $1\leq r,s \leq \infty$, if there exists a constant $C<\infty$, such that for any function $f\in L^{1,s}(\Omega)$, the $(r,s)$-Sobolev-Poincar\'e inequality
\[
\inf\limits_{c \in \mathbb R}\|f-c\|_{L^r(\Omega)}\leq C\|\nabla f\|_{L^s(\Omega)}
\]
holds. We denote by $B_{r,s}(\Omega)$ the best constant in this inequality.

The weak quasiconformal mappings permits us to "transfer" the Sobolev-Poincar\'e inequalities from
one domain to another. In the work \cite{GU19}, the authors obtained the following result.

\begin{theorem}\label{SPineq}
Let a bounded domain $\Omega \subset \mathbb R^n$ be a $(r,s)$-Sobolev-Poincar\'e domain, $1<s \leq r<\infty$, and there exists a weak $(p,s)$-quasiconformal homeomorphism $\varphi:\Omega\to\widetilde{\Omega}$ of a domain $\Omega$ onto a bounded domain $\widetilde{\Omega}$, possesses the Luzin $N$-property (an image of a set of measure zero has measure zero) and such that
\[
M_{r,q}(\Omega)=\biggr(\int_{\Omega}\left|J(x,\varphi)\right|^{\frac{r}{r-q}}\, dx\biggr)^{\frac{r-q}{rq}}< \infty
\]
for some $1\leq q<r$. Then in the domain $\widetilde{\Omega}$ the $(q,p)$-Sobolev-Poincar\'e inequality
\[
\inf\limits_{c \in \mathbb R}\biggr(\int_{\widetilde{\Omega}}|f(y)-c|^q dy\biggr)^{\frac{1}{q}} \leq B_{q,p}(\widetilde{\Omega})
\biggr(\int_{\widetilde{\Omega}}|\nabla f(y)|^p dy\biggr)^{\frac{1}{p}},\,\,\, f\in W^{1,p}(\widetilde{\Omega}),
\]
holds and for $1<s<p$, we have
\[
B_{q,p}(\widetilde{\Omega})\leq K_{p,s}(\varphi;\Omega)M_{r,q}(\Omega)B_{r,s}(\Omega).
\]
Here $B_{r,s}(\Omega)$ is the best constant in the $(r,s)$-Sobolev-Poincar\'e inequality in the
domain $\Omega$ and $K_{p,s}(\varphi;\Omega)$ is as defined in \eqref{kps}.
\end{theorem}

Recall that $\Omega_{\gamma}$ are bounded domains with anisotropic H\"older $\gamma$-singularities (introduced in \cite{GG94}):
\[
\Omega_{\gamma}=\left\{x=(x_1,x_2,\ldots,x_n)\in \mathbb R^n:0<x_n<1, 0<x_i<g_i(x_n), i=1,2,\ldots,n-1\right\},
\]
where $g_i(t)=t^{\gamma_i}$, $\gamma_i\geq 1$, $0 < t < 1$ are H\"older functions and for the function
$G=\prod\limits_{i=1}^{n-1}g_i$ denote by
\[
\gamma = \frac{\log G(t)}{\log t}+1,\,\, 0<t<1.
\]
It is evident that $\gamma\geq n$. In the case $g_1=g_2=\ldots=g_{n-1}$ we will say that domain $\Omega_{\gamma}$ is a domain with $\sigma$-H\"older singularity, $\sigma=(\gamma -1)/(n-1)$. For $g_1(t)=g_2(t)=\ldots=g_{n-1}(t)=t$ we will use notation $\Omega_1$ instead of $\Omega_{\gamma}$.

Define the mapping $\varphi_a:\Omega_1\to \Omega_{\gamma}$, $a>0$, by
\[
\varphi_a(x)=\left(\frac{x_1}{x_n}g_1^a(x_n),\ldots,\frac{x_{n-1}}{x_n}g_{n-1}^a(x_n),x^a_n\right).
\]

Let us formulate  the following result from  the work \cite{GU19} in the refined form because of the very useful remark by Charles Fefferman.

\begin{theorem}\label{S-P-ineq}
Let $(n-p)/(\gamma -p)<a<p(n-s)/s(\gamma -p)$. Then the mapping $\varphi_a:\Omega_1\to \Omega_{\gamma}$ is a weak (p,s)-quasiconformal mapping,
$1<s<p<\gamma$, $1<s<n$, from the Lipschitz convex domain $\Omega_1$ onto the ``cusp'' domain $\Omega_{\gamma}$ with
\[
K_{p,s}(\varphi_a;\Omega_1)\leq a^{-\frac{1}{p}} A(p,s,\gamma) \sqrt{\sum_{i=1}^{n-1}(a\gamma_i-1)^2+n-1+a^2},
\]
where $A(p,s,\gamma)=\left(\frac{p-s}{np-s(a(\gamma-p)+p)}\right)^{(p-s)/ps}$. 
\end{theorem}

By Theorem~\ref{thmemb} the embedding operator
$$
W^{1,p}(\Omega_{\gamma})\hookrightarrow L^{q}(\Omega_{\gamma})
$$
is compact for any $1<q<p^*_{\gamma}$ and by Theorem \ref{minthm} the first non-trivial Neumann eigenvalue $\lambda_{p,q}(\Omega_{\gamma})$ can be characterized as
\begin{equation*}
\lambda_{p,q}(\Omega_{\gamma})\\=\inf \left\{\frac{\|\nabla u\|_{L^p(\Omega_{\gamma})}^p}{\|u\|_{L^q(\Omega_{\gamma})}^p} : u \in W^{1,p}(\Omega_{\gamma}) \setminus \{0\},
\int_{\Omega_{\gamma}} |u|^{q-2}u~dx=0 \right\}.
\end{equation*}

Furthermore, $\lambda_{p,q}(\Omega_{\gamma})^{-\frac{1}{p}}$ is equal to the best constant $B_{q,p}(\Omega_{\gamma})$ in the $(q,p)$-Sobolev-Poincar\'e inequality
\begin{equation*}
\inf\limits_{c \in \mathbb R}\left(\int_{\Omega_{\gamma}} |f(x)-c|^q~dx\right)^{\frac{1}{q}}\\
\leq B_{q,p}(\Omega_{\gamma})
\left(\int_{\Omega_{\gamma}} |\nabla f(x)|^p~dx\right)^{\frac{1}{p}}, \,\,\, f \in W^{1,p}(\Omega_{\gamma}).
\end{equation*}

We are ready to prove spectral estimates in cusp domains.

\begin{theorem}\label{eigen}
Let
\[
\Omega_{\gamma}:=\left\{x=(x_1,x_2,\ldots,x_n)\in\mathbb R^n:n \geq 3, 0<x_n<1, 0<x_i<x_n^{\gamma_i}, i=1,2,\ldots,n-1\right\}
\]
$\gamma_i\geq 1$, $\gamma :=1+\sum_{i=1}^{n-1}\gamma_i$, be domains with anisotropic H\"older
$\gamma$-singularities.

Then for $1<s<p<\gamma$ and $1<q<p_{\gamma}^*$, we have
\[
\frac{1}{\lambda_{p,q}(\Omega_{\gamma})} \leq \inf\limits_{a \in I_a} a^{\frac{p}{q}-1}
\left(\sum_{i=1}^{n-1}(a\gamma_i-1)^2+n-1+a^2\right)^{\frac{p}{2}} B_{r,s}^p(\Omega_1) A^p(p,s,\gamma)B^p(r,q,\gamma),
\]
where $I_a=\left(n/\gamma,p(n-s)/s(\gamma -p)\right)$ and $s<n$ is chosen such that $p_{\gamma}^*<\frac{ns}{n-s}$. $B_{r,s}(\Omega_1)$ is the best constant in the $(r,s)$-Sobolev-Poincar\'e inequality in the domain $\Omega_1$, $1<q< r<\infty$,
$A(p,s,\gamma)=\left({(p-s)}/{np-s(a(\gamma-p)+p)}\right)^{(p-s)/ps}$, $B(r,q,\gamma)=\left({(r-q)}/{r\gamma a-nq}\right)^{(r-q)/rq}$.
\end{theorem}

\begin{proof}
By Theorem~\ref{S-P-ineq} there exists a locally Lipschitz homeomorphism $\varphi_a:\Omega_1\to \Omega_{\gamma}$, $(n-p)/(\gamma -p)<a<p(n-s)/s(\gamma -p)$,
\[
\varphi_a(x)=\left(\frac{x_1}{x_n}g_1^a(x_n),\ldots,\frac{x_{n-1}}{x_n}g_{n-1}^a(x_n),x^a_n\right)
\]
which maps the convex Lipschitz domain $\Omega_1$ onto the cusp domain $\Omega_{\gamma}$, possesses the Luzin $N$-property, and it is a weak $(p,s)$-quasiconformal mapping, $1<s<p<\gamma$.

Let us check conditions of Theorem~\ref{SPineq}. Since $\varphi_a$ is a weak $(p,s)$-quasiconformal mapping, by Theorem \ref{S-P-ineq}, the constant $K_{p,s}(\varphi_a;\Omega_1)$ is finite and satisfy the estimate
\begin{equation}\label{est1}
K_{p,s}(\varphi_a;\Omega_1)\leq a^{-\frac{1}{p}} A(p,s,\gamma)\sqrt{\sum_{i=1}^{n-1}(a\gamma_i-1)^2+n-1+a^2}.
\end{equation}
The domain $\Omega_1$ is a Lipschitz domain and so is a $(r,s)$-Sobolev-Poincar\'e domain, i.e. $B_{r,s}(\Omega_1)<\infty$.

Let $1<q< r<\infty$. Next we estimate the quantity $M_{r,q}(\Omega_1)$.
\begin{multline*}
M_{r,q}(\Omega_1)= \biggr(\int_{\Omega_1}\left|J(x,\varphi_a)\right|^{\frac{r}{r-q}}\, dx\biggr)^{\frac{r-q}{rq}}
= a^{\frac{1}{q}}
\biggr(\int_{\Omega_1}\left(x_n^{a\gamma -n}\right)^{\frac{r}{r-q}}\, dx\biggr)^{\frac{r-q}{rq}} \\
= a^{\frac{1}{q}}
\biggr(\int_{0}^1\left(x_n^{a\gamma -n}\right)^{\frac{r}{r-q}}
\biggr(\int_{0}^{x_n}\,dx_1 \ldots \int_{0}^{x_n}\,dx_{n-1}\biggr)
dx_n\biggr)^{\frac{r-q}{rq}} \\
= a^{\frac{1}{q}}
\biggr(\int_{0}^1\left(x_n^{a\gamma -n}\right)^{\frac{r}{r-q}} \cdot x_n^{n-1} dx_n\biggr)^{\frac{r-q}{rq}}
= a^{\frac{1}{q}}
\left(\frac{r-q}{r\gamma a-nq}\right)^{\frac{r-q}{rq}} 
<\infty,
\end{multline*}
if
\[
\frac{(a\gamma -n)r}{r-q}+n-1>-1,\,\,\text{i.e.}\,\, a>\frac{nq}{\gamma r}.
\]

Recall that $r>q$. Therefore, since $0<x_n<1$, we obtain
\begin{equation}\label{est2}
M_{r,q}(\Omega_1)\leq a^{\frac{1}{q}}\,\,\text{ if }\,\, a>\frac{n}{\gamma }>\frac{nq}{\gamma r}.
\end{equation}
So, the conditions of Theorem~\ref{SPineq} are fulfilled. Therefore, noting that $1<q<p_{\gamma}^*$, we have $\lambda_{p,q}^{-1}(\Omega_{\gamma})=B_{q,p}^p(\Omega_{\gamma})$ and by Theorem~\ref{SPineq}, we obtain for $1<s<p<\gamma$ the following estimate
\begin{multline*}
\frac{1}{\lambda_{p,q}(\Omega_{\gamma})}\leq K_{p,s}^p(\varphi_a;\Omega_1)M_{r,q}^p(\Omega_1)B^p_{r,s}(\Omega_1) \\
\leq a^{\frac{p}{q}-1}
\left(\sum_{i=1}^{n-1}(a\gamma_i-1)^2+n-1+a^2\right)^{\frac{p}{2}} B_{r,s}^p(\Omega_1) A^p(p,s,\gamma)B^p(r,q,\gamma),
\end{multline*}
where $\max\left\{(n-p)/(\gamma -p),{n}/{\gamma}\right\}<a<p(n-s)/s(\gamma -p)$ and $B_{r,s}(\Omega_1)$ is the best constant in the $(r,s)$-Sobolev-Poincar\'e inequality in the domain $\Omega_1$. To deduce the last inequality above, we have also used the estimates \eqref{est1} and \eqref{est2}. Finally remark, that $(n-p)/(\gamma -p)<{n}/{\gamma}$ and so $\max\left\{(n-p)/(\gamma -p),{n}/{\gamma}\right\}={n}/{\gamma}$.
\end{proof}

\begin{remark}\label{rmk}
The estimate of the constant in the $(r,s)$-Sobolev-Poincar\'e inequality in the domain $\Omega_1$ was obtained in \cite{GU17}:
\[
B_{r,s}(\Omega_1)\leq n\left(\frac{1-\delta}{1/n-\delta}\right)^{1-\delta}\omega_n^{1-\frac{1}{n}}
\left(\frac{1}{(n+1)!}\right)^{\frac{1}{n}-\delta},\,\, \delta=\frac{1}{s}-\frac{1}{r}\geq 0.
\]

The problem of exact value of constants in the $(r,s)$-Sobolev-Poincar\'e inequality in the case $s \neq r$ is a complicated open problem even in the case of the unit disk $\mathbb D \subset \mathbb R^2$.
\end{remark}

\begin{corollary}
Let us consider an application of Theorem~\ref{eigen} to the following spectral problem
\begin{equation}\label{ex}
-\text{div}(|\nabla u|\nabla u)=\lambda\|u\|_{L^2(\Omega_{\gamma})}\cdot u\,\,\text{ in }\,\,\Omega_{\gamma}\subset\mathbb R^3,
\end{equation}
with the Neumann boundary condition in an anisotropic H\"older $\gamma$-singular domains $\Omega_{\gamma}$, where
\[
\Omega_{\gamma}:=\left\{x=(x_1,x_2,x_3)\in\mathbb R^3: 0<x_3<1, 0<x_i<x_3^{\gamma_i}, i=1,2\right\},
\]
$\gamma_i\geq 1$, $\gamma :=1+\gamma_1+\gamma_2$. Then the first non-trivial Neumann $(3,2)$-eigenvalue of the spectral problem (\ref{ex}) satisfies the estimate:
\begin{multline*}
\frac{1}{\lambda_{3,2}(\Omega_{\gamma})} \\
\leq \inf\limits_{\frac{3}{\gamma}<a < \frac{3(3-s)}{s(\gamma -3)}} \sqrt{a}
\left((a\gamma_1-1)^2+(a\gamma_2-1)^2+2+a^2\right)^{\frac{3}{2}} B_{r,s}^3(\Omega_1) A^3(3,s,\gamma)B^3(r,2,\gamma).
\end{multline*}
Let us take $r=5/2$, $s=3/2$ then by \cite{GU17}
\[
B_{\frac{5}{2},\frac{3}{2}}(\Omega_1) \leq 3 \cdot 11^{\frac{11}{15}} \left(\frac{4 \pi}{3}\right)^{\frac{2}{3}}
\left(\frac{1}{24}\right)^{\frac{1}{15}} \approx 12\pi.
\]
If $3<\gamma<5$ we can take $a=1$ and so 
\[
\frac{1}{\lambda_{3,2}(\Omega_{\gamma})} \leq (12\pi)^3
\left((\gamma_1-1)^2+(\gamma_2-1)^2+3\right)^{\frac{3}{2}},
\]
since in this case $A^3(3,3/2,\gamma)\cdot B^3(5/2,2,\gamma)<1$.
\end{corollary}

\section{Existence and regularity results}

Throughout this section, we assume that $1<p<\gamma$ and $q=2$ unless otherwise stated. Let
$$
X:=\Big\{u\in W^{1,p}(\Omega_{\gamma}):\int_{\Omega_{\gamma}}u\,dx=0\Big\}.
$$

By Theorem~\ref{thmemb}, we endow the norm $\|\cdot\|_{X}$ on $X$ defined by
\begin{equation}\label{mfn}
\|u\|_{X}=\left(\int_{\Omega_{\gamma}}|\nabla u|^p\,dx\right)^\frac{1}{p}.
\end{equation}

The Lebesgue space $Y:=L^2(\Omega_{\gamma})$ be endowed with the norm
\begin{equation}\label{lqn}
\|u\|_{Y}:=\left(\int_{\Omega_{\gamma}}|u|^2\,dx\right)^\frac{1}{2}.
\end{equation}

\subsection{Regularity results}

Let us formulate the regularity results for the Neumann $(p,q)$-eigenvalue problem.

\begin{theorem}\label{newthm}
Let $1<p<\gamma$. Then the following properties hold:
\vskip 0.2cm
\noindent
$(a)$ There exists a sequence $\{w_n\}_{n\in\mathbb{N}}\subset X\cap Y$ such that $\|w_n\|_{Y}=1$ and for every $v\in X$, we have
\begin{equation}\label{its}
\int_{\Omega_{\gamma}}|\nabla\,w_{n+1}|^{p-2}\nabla\,w_{n+1} \nabla\,v\,dx=\mu_n\int_{\Omega_{\gamma}}w_{n}v\,dx,
\end{equation}
where
\begin{equation*}\label{subopmin}
\mu_n\geq \lambda:=\inf\left\{\int_{\Omega_{\gamma}}|\nabla\,u|^p\,dx:u\in X\cap {Y},\,\|u\|_{Y}=1\right\}.
\end{equation*}
\vskip 0.2cm
\noindent
(b) Moreover, the sequences $\{\mu_n\}_{n\in\mathbb{N}}$ and $\{\|w_{n+1}\|_{X}^{p}\}_{n\in\mathbb{N}}$ given by \eqref{its} are nonincreasing and converge to the same limit $\mu$, which is bounded below by $\lambda$. Further, there exists a subsequence $\{n_j\}_{j\in\mathbb{N}}$ such that both $\{w_{n_j}\}_{j\in\mathbb{N}}$ and $\{w_{n_{j+1}}\}_{j\in\mathbb{N}}$ converges in $X$ to the same limit $w\in X\cap Y$ with $\|w\|_{Y}=1$ and $(\mu,w)$ is an eigenpair of \eqref{meqn}.
\end{theorem}

\begin{theorem}\label{subopthm1}
Let $1<p<\gamma$ and $q=2$. Suppose $\{u_n\}_{n\in\mathbb{N}}\subset X\cap Y$ is a minimizing sequence for $\lambda$, that is $\|u_n\|_Y=1$ and $\|u_n\|_X^{p}\to\lambda$. Then there exists a subsequence $\{u_{n_j}\}_{j\in\mathbb{N}}$ which converges weakly in $X$ to $u\in X\cap Y$ such that
$
\lambda=\|u\|^p_{X}.
$
Moreover, $u$ is an eigenfunction of \eqref{maineqn} corresponding to $\lambda$ and its associated eigenfunctions are precisely the scalar multiple of those vectors at which $\lambda$ is reached.
\end{theorem}

Moreover, we have the following regularity results.
\begin{theorem}\label{regthm1}
Let $1<p<\gamma$ and $q=2$. Assume that $\lambda>0$ is an eigenvalue of the problem \eqref{maineqn} and $u\in X\setminus\{0\}$ is a corresponding eigenfunction. Then \\
(i) $u\in L^\infty(\Omega_{\gamma})$. \\
(ii) Moreover, if $u\in X\setminus\{0\}$ is nonnegative in $\Omega_{\gamma}$, then $u>0$ in $\Omega_{\gamma}$. Further, for every $\omega\Subset \Omega_{\gamma}$ there exists a positive constant $c$ depending on $\omega$ such that $u\geq c>0$ in $\omega$.
\end{theorem}

\begin{remark}
Theorems~\ref{newthm}--\ref{regthm1} are correct in the case of Lipschitz domains $\Omega\subset\mathbb R^n$. In this case $p^{*}_{n}=np/(n-p)$, where $1<p<n$.
\end{remark}

\subsection{Operators associated to eigenvalue problems}

Let us denote by $X^*$ and $Y^*$ the dual of $X$ and $Y$ respectively.
Let us define the operator $A:X\to X^*$ by
\begin{equation}\label{ma}
\begin{split}
\langle A(v),w\rangle&=\int_{\Omega_{\gamma}}|\nabla v|^{p-2}\nabla v\nabla w\,dx,\quad \forall v,w\in X
\end{split}
\end{equation}
and $B:Y\to Y^*$ by
\begin{equation}\label{mb}
\begin{split}
\langle B(v),w\rangle&=\int_{\Omega_{\gamma}}v w\,dx,\quad \forall v,w\in Y.
\end{split}
\end{equation}
First we state some useful results. The following result from \cite[Theorem $9.14$]{var} will be useful for us.
\begin{theorem}\label{MBthm}
Let $V$ be a real separable reflexive Banach space and $V^{*}$ be the dual of $V$. Assume that $A:V\to V^{*}$ is a bounded, continuous, coercive and monotone operator. Then $A$ is surjective, i.e., given any $f\in V^{*}$, there exists $u\in V$ such that $A(u)=f$. If $A$ is strictly monotone, then $A$ is also injective.
\end{theorem}
Moreover, for the following algebraic inequality, see \cite[Lemma $2.1$]{Dama}.
\begin{lemma}\label{alg}
Let $1<p<\infty$. Then for any $a,b\in\mathbb{R}^N$, there exists a positive constant $C=C(p)$ such that
\begin{equation}\label{algineq}
(|a|^{p-2}a-|b|^{p-2}b, a-b)\geq
C(|a|+|b|)^{p-2}|a-b|^2.
\end{equation}
\end{lemma}

Next, we prove the following result.
\begin{lemma}\label{newlem}
$(i)$ The operators $A$ defined by \eqref{ma} and $B$ defined by \eqref{mb} are continuous. $(ii)$ Moreover, $A$ is bounded, coercive and monotone.
\end{lemma}
\begin{proof}
\noindent
$(i)$ \textbf{Continuity:} We only prove the continuity of $A$, since the continuity of $B$ would follow similarly. To this end, suppose $v_n\in X$ such that $v_n\to v$ in the norm of $X$. Thus, up to a subsequence $\nabla v_{n}\to \nabla v$ in $\Omega_{\gamma}$. We observe that
\begin{equation}\label{mfd}
\||\nabla v_{n}|^{p-1}\|_{L^\frac{p}{p-1}(\Omega_{\gamma})}\leq c,
\end{equation}
for some constant $c>0$, which is independent of $n$. Thus, up to a subsequence, we have
\begin{equation}\label{fc}
|\nabla v_{n}|^{p-2}\nabla v_{n}{\rightharpoonup} |\nabla v|^{p-1}\nabla v\text{ weakly in }L^{p'}(\Omega_{\gamma}).
\end{equation}
Since, the weak limit is independent of the choice of the subsequence, as a consequence of \eqref{fc}, we have
$$
\lim_{n\to\infty}\langle Av_n,w\rangle=\langle Av,w\rangle
$$
for every $w\in X$. Thus $A$ is continuous.
\vskip 0.2cm
\noindent
$(ii)$ \textbf{Boundedness:} Using the estimate \eqref{mest}, we have
$$
\|Av\|_{X^*}=\sup_{\|w\|_X\leq 1}|\langle Av,w\rangle|\leq\|v\|_X^{p-1}\|w\|_X\leq\|v\|^{p-1}_X.
$$
Thus, $A$ is bounded.

\noindent
\textbf{Coercivity:}  We observe that
$$
\langle Av,v\rangle=\|v\|_{X}^p.
$$
Since $p>1$, we have $A$ is coercive.

\noindent
\textbf{Monotonicity:} Using Lemma \ref{alg}, it follows that there exists a constant $C=C(p)>0$ such that
for every $v,w\in X$, we have
\begin{multline*}
\langle Av-Aw,v-w\rangle=\int_{\Omega_{\gamma}}(|\nabla\,v|^{p-2}\nabla\,v-|\nabla_{\textrm{H}}\,w|^{p-2}\nabla\,w,\nabla\,(v-w))\,dx\\
=
\int_{\Omega_{\gamma}}(|\nabla\,v|^{p-2}\nabla\,v-|\nabla\,w|^{p-2}\nabla\,w,\nabla\,v-\nabla\,w)\,dx
\\
 \geq C(p)\int_{\Omega_{\gamma}}(|\nabla\,v|+|\nabla\,w|)^{p-2}|\nabla\,v-\nabla\,w|^2\,dx
\geq 0.
\end{multline*}
Thus, $A$ is a monotone operator.
\end{proof}

\begin{lemma}\label{auxlmab}
The operators $A$ defined by \eqref{ma} and $B$ defined by \eqref{mb} satisfy the following properties:
\vskip 0.2cm
\noindent
$(H_1)$ $A(tv)=|t|^{p-2}tA(v)\quad\forall t\in\R\quad \text{and}\quad\forall v\in X$.
\vskip 0.2cm

\noindent
$(H_2)$ $B(tv)=tB(v)\quad\forall t\in\R\quad \text{and}\quad\forall v\in Y$.
\vskip 0.2cm

\noindent
$(H_3)$ $\langle A(v),w\rangle\leq\|v\|_X^{p-1}\|w\|_X$ for all $v,w\in X$, where the equality holds if and only if $v=0$ or $w=0$ or $v=t w$ for some $t>0$.
\vskip 0.2cm

\noindent
$(H_4)$ $\langle B(v),w\rangle\leq\|v\|_Y^{p-1}\|w\|_Y$ for all $v,w\in Y$, where the equality holds if and only if $v=0$ or $w=0$ or $v=t w$ for some $t\geq 0$.
\vskip 0.2cm

\noindent
$(H_5)$ For every $w\in Y\setminus\{0\}$ there exists $u\in X\setminus\{0\}$ such that
$$
\langle A(u),v\rangle=\langle B(w),v\rangle\quad\forall\quad v\in X.
$$
\end{lemma}
\begin{proof}
\vskip 0.2cm

\noindent
$(H_1)$ Follows by the definition of $A$ in \eqref{ma}.

\noindent
$(H_2)$  Follows by the definition of $B$ in \eqref{mb}.
\vskip 0.2cm

\noindent
$(H_3)$ First using Cauchy-Schwartz inequality and then by H\"older's inequality with exponents $\frac{p}{p-1}$ and $p$, for every $v,w\in X$, we obtain
\begin{equation}\label{mest}
\begin{split}
\langle Av,w\rangle&=\int_{\Omega_{\gamma}}|\nabla v|^{p-2}\nabla v\nabla w\,dx\leq\int_{\Omega_{\gamma}}|\nabla v|^{p-1}|\nabla w|\,dx\\
&\leq\Big(\int_{\Omega_{\gamma}}|\nabla v|^p\,dx\Big)^\frac{p-1}{p}\Big(\int_{\Omega_{\gamma}}|\nabla w|^p\,dx\Big)^\frac{1}{p}=\|v\|_X^{p-1}\|w\|_X.
\end{split}
\end{equation}
Let the equality
\begin{equation}\label{mequal}
\langle A(v),w\rangle=\|v\|_X^{p-1}\|w\|_X
\end{equation}
holds for every $v,w\in X$. We claim that either $v=0$ or $w=0$ or $v=tw$ for some constant $t>0$. Indeed, if $v=0$ or $w=0$, this is trivial. Therefore, we assume $v\neq 0$ and $w\neq 0$ and prove that $v=tw$ for some constant $t>0$. By the estimate \eqref{mest} if the equality \eqref{mequal} holds, then we have
\begin{equation}\label{mequal1}
\langle A(v), w\rangle=\int_{\Omega_{\gamma}}|\nabla v|^{p-1}|\nabla w|\,dx,
\end{equation}
which gives  us
\begin{equation}\label{mCS}
\int_{\Omega_{\gamma}}f(x)\,dx=0,
\end{equation}
where
$$
f(x)=|\nabla v|^{p-1}|\nabla w|-|\nabla v|^{p-2}\nabla v\nabla w.
$$
By Cauchy-Schwartz inequality, we have $f\geq 0$ in $\Omega_{\gamma}$. Hence using this fact in \eqref{mCS}, we have $f=0$ in $\Omega_{\gamma}$, which reduces to
\begin{equation}\label{mfCS}
|\nabla v|^{p-1}\nabla w=|\nabla v|^{p-2}\nabla v\nabla w\text{ in }\Omega_{\gamma},
\end{equation}
which gives $\nabla{v}(x)=c(x)\nabla w(x)$ for some $c(x)\geq 0$.

On the other hand, if the equality \eqref{mequal} holds, then by the estimate \eqref{mest} we have
\begin{equation}\label{mequal2}
\begin{split}
f_1=f_2,
\end{split}
\end{equation}
where
$$
f_1=\int_{\Omega_{\gamma}}|\nabla{v}|^{p-1}\nabla w\,dx,\quad
f_2=\left(\int_{\Omega_{\gamma}}|\nabla{v}|^p\,dx\right)^\frac{p-1}{p}\left(\int_{\Omega_{\gamma}}|\nabla w|^p\,dx\right)^\frac{1}{p}.
$$
Thus, we have $|\nabla{v}|(x)=d|\nabla w|(x)$ a.e. $x\in \Omega_{\gamma}$ for some constant $d>0$. Thus, we obtain $c(x)=d$. As a consequence, we get $\nabla{v}=d\nabla w$ a.e. in $\Omega_{\gamma}$ and therefore, we deduce that $\|v-dw\|_X=0$, which gives $v=dw$ a.e. in $\Omega_{\gamma}$. Hence, the property $(H3)$ is verified.
\vskip 0.2cm

\noindent
$(H_4)$ The hypothesis $(H_4)$ can be verified similarly.

\vskip 0.2cm

\noindent
$(H_5)$ We observe that $X$ is a separable and reflexive Banach space. By Lemma \ref{newlem}, the operator $A:X\to X^*$ is bounded, continuous, coercive and monotone.

By the Sobolev embedding theorem, we have $X$ is continuously embedded in $Y$. Therefore, $B(w)\in X^*$ for every $w\in Y\setminus\{0\}$.

Hence, by Theorem \ref{MBthm}, for every $w\in Y\setminus\{0\}$, there exists $u\in X\setminus\{0\}$ such that
$$
\langle A(u),v\rangle=\langle B(w),v\rangle\quad\forall v\in X.
$$
Hence the property $(H_5)$ holds. This completes the proof.
\end{proof}

\subsection{Proof of the regularity results:}
\vskip 0.2cm
\noindent
\textbf{Proof of Theorem \ref{newthm}:}
\vskip 0.2cm
\noindent
$(a)$ First we recall the definition of the operators $A:X\to X^*$ from \eqref{ma} and $B:Y\to Y^*$ from \eqref{mb} respectively. Then, noting the property $(H_5)$ from Lemma \ref{auxlmab} and proceeding along the lines of the proof in \cite[page $579$ and pages $584-585$]{Ercole}, the result follows.
\vskip 0.2cm
\noindent
$(b)$ We observe that $X$ is uniformly convex Banach space and by the Sobolev embedding theorem \cite{GG94,GU09}, $X$ is compactly embedded in $Y$. Next, using Lemma \ref{newlem}-$(i)$, the operators $A:X\to X^*$ and $B:Y\to Y^*$ are continuous and by Lemma \ref{auxlmab}, the properties $(H_1)-(H_5)$ holds. Noting these facts, the result follows from \cite[page $579$, Theorem 1]{Ercole}. \qed
\vskip 0.2cm
\noindent
\textbf{Proof of Theorem \ref{subopthm1}:} The proof follows due to the same reasoning as in the proof of Theorem \ref{newthm}-$(b)$ except that here we apply \cite[page $583$, Proposition $2$]{Ercole} in place of \cite[page $579$, Theorem 1]{Ercole}.
\vskip 0.2cm
\noindent
\textbf{Proof of Theorem \ref{regthm1}:}
\vskip 0.2cm
\noindent
$(i)$ Due to the homogeneity of the equation \eqref{maineqn}, without loss of generality, we assume that $\|u\|_{Y}=1$. Let $k\geq 1$ and set $A(k):=\{x\in \Omega_{\gamma}:u(x)>k\}$. Choosing $v=(u-k)^+$ as a test function in \eqref{mwksol}, we obtain
\begin{multline}\label{regtst2}
\int_{A(k)}|\nabla u|^p\,dx=\lambda\int_{A(k)}u(u-k)\,dx \\
\leq\lambda\int_{A(k)}|u|(u-k)\,dx\leq \lambda\int_{A(k)}|u|(u-k)\,dx.
\end{multline}
Noting \eqref{regtst2} and proceeding along the lines of the proof of \cite[Theorem 3.4, Pages 7--8]{GUsub}, the result follows.

\vskip 0.2cm
\noindent
$(ii)$ By \cite[Theorem $1.2$]{Tru}, the result follows.

\vskip 0.2cm
\noindent
\textbf{Acknowledgements.}

\noindent
The authors thank Pier Domenico Lamberti  for very useful and fruitful discussions and remarks on the topic.

\noindent
The second author was supported by RSF Grant No. 23-21-00080.

\vskip 0.2cm
\noindent
\textbf{Data availability statements.}
\noindent
Data sharing not applicable to this article as no datasets were generated or analysed during the current study.

\vskip 0.3cm

Department of Mathematical Sciences, Indian Institute of Science Education and Research Berhampur, Berhampur, Odisha 760010, India

Department of Mathematics, Indian Institute of Technology Indore, Khandwa Road, Simrol, Indore 453552, India

\emph{E-mail address:} \email{pgarain92@gmail.com} \\

 Regional Scientific and Educational Mathematical Center, Tomsk State University, 634050 Tomsk, Lenin Ave. 36, Russia
							
\emph{E-mail address:} \email{va-pchelintsev@yandex.ru}   \\
			
Department of Mathematics, Ben-Gurion University of the Negev, P.O.Box 653, Beer Sheva, 8410501, Israel
							
\emph{E-mail address:} \email{ukhlov@math.bgu.ac.il}

\end{document}